\documentclass{amsart}

\usepackage{xcolor}

\usepackage{amsmath, amsthm, amssymb}
\usepackage{thmtools}
\usepackage{mathtools}
\mathtoolsset{centercolon}
\usepackage{enumitem}
\usepackage{graphicx}
\usepackage{subdepth}
\usepackage{delimset}
\usepackage{nicematrix}
\usepackage{tikz}
\usetikzlibrary{calc}
\usepackage[style=alphabetic]{biblatex}
\AtBeginBibliography{\small}
\usepackage{hyperref}
\hypersetup{
  pdftitle={Matrix splitting and ideals in B(H)},
  pdfauthor={Jireh Loreaux, Gary Weiss}
}
\usepackage{cleveref}

\newcommand{\exchangesymbols}[2]{%
  \expandafter\mathchardef\expandafter#1\number\expandafter#2\expandafter\relax
\expandafter\mathchardef\expandafter#2\number#1%
}

\exchangesymbols{\phi}{\varphi}
\exchangesymbols{\epsilon}{\varepsilon}
\exchangesymbols{\emptyset}{\varnothing}

\DeclareFontFamily{U}{mathb}{\hyphenchar\font45}
\DeclareFontShape{U}{mathb}{m}{n}{
<-6> mathb5 <6-7> mathb6 <7-8> mathb7
<8-9> mathb8 <9-10> mathb9
<10-12> mathb10 <12-> mathb12
}{}
\DeclareSymbolFont{mathb}{U}{mathb}{m}{n}
\DeclareMathSymbol{\pprec}{\mathrel}{mathb}{"CE}
\DeclareMathSymbol{\ssucc}{\mathrel}{mathb}{"CF}

\makeatletter
\newcommand\bigcdot{\mathpalette\bigcdot@{.5}}
\newcommand*\bigcdot@[2]{\mathbin{\vcenter{\hbox{\scalebox{#2}{$\m@th#1\bullet$}}}}}
\makeatother

\newcommand{\term}[1]{\emph{#1}}

\newcommand{\I}{\mathcal{I}}
\newcommand{\B}{\mathcal{B}}
\newcommand{\Hil}{\mathcal{H}}
\newcommand{\K}{\mathcal{K}}
\newcommand{\finrank}{\mathcal{F}}
\newcommand{\traceclass}{\mathcal{L}_1}
\newcommand{\hilbertschmidt}{\mathcal{L}_2}

\newcommand{\charset}{\Sigma}

\newcommand{\cz}{c_0}

\newcommand{\czstar}{\cz^{*}}

\newcommand{\nats}{\mathbb{N}}
\newcommand{\ints}{\mathbb{Z}}

\newcommand{\complex}{\mathbb{C}}

\newcommand{\ideal}{\delim<>}
\newcommand{\paren}{\delim()}
\newcommand{\set}{\delim\{\}}
\newcommand{\setb}{\delimpair\{{[m]\vert}\}}
\newcommand{\innerprod}{\delimpair<{[.],}>}
\newcommand{\seq}{\delim()}
\newcommand{\ceil}{\delim\lceil\rceil}
\newcommand{\matrep}{\delim()}
\newcommand{\amclosure}[1]{\closure[am]{#1}}
\newcommand{\ep}{E_{\mathcal{P}}}

\newcommand{\conj}[1]{\overline{#1}}
\newcommand{\sph}{\mathbb{S}}

\DeclareMathOperator{\diag}{diag}
\DeclareMathOperator{\rank}{rank}
\DeclareMathOperator{\trace}{Tr}
\DeclareMathOperator{\spans}{span}

\newcommand{\closure}[2][]{\overline{#2}^{#1}}
\DeclareMathOperator{\range}{ran}

\setlist[enumerate]{label=\textup{(\roman*)}}

\newcommand{\fnmark}[1]{%
  \footnotemark%
  \newcounter{#1}%
  \setcounter{#1}{\value{footnote}}%
}

\theoremstyle{plain} 
\newtheorem{theorem}{Theorem}

\theoremstyle{definition}
\newtheorem{definition}[theorem]{Definition}
\newtheorem{example}[theorem]{Example}

\theoremstyle{remark}
\newtheorem*{remark}{Remark}

\newcounter{case}
\makeatletter\@addtoreset{case}{theorem}\makeatother

\title{Matrix splitting and ideals in $\B(\Hil)$}

\author{Jireh Loreaux}
\email{jloreau@siue.edu}

\author{Gary Weiss$^{\dagger}$}
\email{gary.weiss@uc.edu}
\thanks{$^{\dagger}$Partially supported by Simons Foundation grants 245014 and 63655}

\begin{document}


\maketitle

\begin{abstract}
  We investigate the relationship between ideal membership of an operator and its pieces relative to several canonical types of partitions of the entries of its matrix representation with respect to a given orthonormal basis.
  Our main theorems establish that if $T$ lies in an ideal $\I$, then $\sum P_n T P_n$ (or more generally $\sum Q_n T P_n$) lies in the arithmetic mean closure of $\I$ whenever $\set{P_n}$ (and also $\set{Q_n}$) is a sequence of mutually orthogonal projections; and in any basis for which $T$ is a block band matrix, in particular, when in Patnaik--Petrovic--Weiss universal block tridiagonal form, then all the sub/super/main-block diagonals of $T$ are in $\I$.
  And in particular, the principal ideal generated by this $T$ is the finite sum of the principal ideals generated by each sub/super/main-block diagonals.
\end{abstract}

\section{Introduction}

In the study of infinite matrix representations of operators in $\B(\Hil)$, and especially the structure of commutators, it is common and natural to split up a target operator $T$ into two (or finite) sum of natural parts.
For example, every finite matrix is the sum of its upper triangular part and its lower triangular part (including the diagonal in either part as you choose).

Formally this obviously holds also for infinite matrices, but not in $\B(\Hil)$.
That is, as is well-known, the upper or lower triangular part of a matrix representation for a bounded operator is not necessarily a bounded operator.
The Laurent operator with zero-diagonal matrix representation $\matrep{\frac{1}{i-j}}_{i \neq j}$ represents a bounded operator but its upper and lower triangular parts represent unbounded operators.
From this we can produce a compact operator whose upper triangular part is unbounded.

\begin{example}
  \label{ex:bdd-upp-triangular-unbdd}
  Consider the zero-diagonal Laurent matrix $\matrep{\frac{1}{i-j}}_{i \neq j}$, which corresponds to the Laurent multiplication operator $M_{\phi} \in \B(L^2(\sph^1))$ where
  \begin{equation*}
    \phi(z) := \sum_{0 \neq n \in \ints} \frac{z^n}{n} = \sum_{n=1}^{\infty} \frac{z^n}{n} - \sum_{n=1}^{\infty} \frac{\conj{z}^n}{n} = \log(1-z) - \log(1-\conj{z}) = \log\paren*{\frac{1-z}{\conj{1-z}}},
  \end{equation*}
  which is bounded since the principle logarithm of a unit modulus function $\phi \in L^{\infty}(\sph^1)$.
  On the other hand, the upper triangular part $\Delta(M_{\phi})$ of $M_{\phi}$ corresponds to multiplication by $\log(1-z) \notin L^{\infty}(\sph^1)$, and is therefore not a bounded operator.
  Additionally, as is well-known, the same boundedness/unboundedness properties are shared by the corresponding Toeplitz operator $T_{\phi}$ and its $\Delta(T_{\phi})$.
  Indeed, this follows from the fact that if $P \in \B(L^2(\sph^1))$ is the projection onto the Hardy space $H^2$, then $P M_{\phi} P$ and $P^{\perp} M_{\phi} P^{\perp}$ are unitarily equivalent, and $P M_{\phi} P^{\perp} = P \Delta(M_{\phi}) P^{\perp}$ is bounded.

  To produce a compact operator whose upper triangular part is unbounded, take successive corners $P_n T_{\phi} P_n$ where $P_n$ is the projection onto $\spans \set{e_1,\ldots,e_n}$.
  Since $\norm{P_n T_{\phi} P_n} \uparrow \norm{T_{\phi}}$ and $\norm{P_n \Delta(T_{\phi}) P_n} \uparrow \infty$, then $\bigoplus_{n=1}^{\infty} \frac{P_n T_{\phi} P_n}{\norm{P_n \Delta(T_{\phi}) P_n}^{1/2}}$ is compact and its upper triangular part is unbounded.
  Similarly, $\bigoplus_{n=1}^{\infty} \frac{P_n T_{\phi} P_n}{\norm{P_n \Delta(T_{\phi}) P_n}}$ is compact but its upper triangular part is bounded and noncompact.
\end{example}

Focusing attention on $\B(\Hil)$ ideals yields a fruitful area of study:
for a Hilbert--Schmidt operator, in any basis, any partition of the entries of its matrix representation has its parts again Hilbert--Schmidt.\footnotemark{}
This leads to a natural question for which the authors are unaware of the answer: is the Hilbert--Schmidt ideal the \emph{only} (nonzero) ideal with this property?

\footnotetext{%
  Of course, for any ideal $\I$ contained within the Hilbert--Schmidt ideal $\hilbertschmidt$, and any $T \in \I$, the upper triangular part $\Delta(T) \in \hilbertschmidt$, but one may wonder if anything stronger can be said.
  In the case of the trace-class ideal $\traceclass$, Gohberg--Krein \cite[Theorem~III.2.1]{GK-1970} showed that $\Delta(T)$, in the terminology of \cite{DFWW-2004-AM}, lies in the arithmetic mean closure of the principal ideal generated by $\diag(\frac{1}{n})$.%
}

For the compact operators $\K(\Hil)$, depending on the shape of the matrix parts for $T$, the problem of determining when its parts are in $\K(\Hil)$ (i.e., ideal invariant) can be a little subtler.
Indeed, as noted in \Cref{ex:bdd-upp-triangular-unbdd}, the upper triangular part of a compact operator may not be compact (nay bounded);
on the other hand, it is well-known and elementary that the diagonal sequence $\seq{d_n}$ of a compact operator converges to zero (i.e., $\diag \seq{d_n}$ is compact), and the same holds for all the sub/super-diagonals as well.
In contrast, this fails for certain matrix representations for a finite rank operator;
that is, the diagonal of a finite rank operator may not be finite rank (e.g., $(\frac{1}{ij})_{i,j \ge 1}$ is rank-1 but its diagonal $\diag(\frac{1}{j^2}) \notin \finrank(\Hil)$).

Here we study this question for general $\B(\Hil)$-ideals: For an ideal $\I$ and all pairs $\set{P_n}, \set{Q_n}$ of sequences of mutually orthogonal projections, when are the generalized diagonals $\sum Q_n T P_n \in \I$ whenever $T \in \I$? (The canonical block diagonals are $\sum P_{n+k} T P_n$ and $\sum P_n T P_{n+k}$.)
We find this especially pertinent in our current search for commutator forms of compact operators \cite{PPW-2021}, growing out of \cite{BPW-2014-VLOT}; and, in view of the second author’s work with V. Kaftal \cite{KW-2011-IUMJ} on diagonal invariance for ideals, useful in recent discoveries by the second author with S. Petrovic and S. Patnaik \cite{PPW-2020-TmloVLOt} on their universal finite-block tridiagonalization for arbitrary $\B(\Hil)$ operators and the consequent work on commutators \cite{PPW-2021}.

Evolution of questions:
\begin{enumerate}
\item For which $\B(\Hil)$-ideals $\I$ does a tridiagonal operator $T$ have its three diagonal parts also in $\I$?
  This question arose from the stronger question: for which tridiagonal operators $T \in \K(\Hil)$ are the diagonals parts in $\ideal{T}$?
  \Cref{thm:bandable} guarantees the latter is always true, even for finite band operators.
\item The same questions but more generally for a block tridiagonal $T$ (see \Cref{def:block-decomposition}) and its three block diagonals (see \Cref{def:shift-representation}).
  Again, \Cref{thm:bandable} guarantees this is always true, and likewise for finite block band operators.
  That is, if
  \(
  T =
  \begin{pNiceMatrix}[small,xdots/line-style=solid]
    B & A & 0 & {} \\[-1em]
    C & \Ddots & \Ddots & \Ddots \\[-1em]
    0 & \Ddots & & {} \\[-1em]
    & \Ddots & & {} \\
  \end{pNiceMatrix}
  \in \I  
  \),
  then
  \(
  \begin{pNiceMatrix}[small,xdots/line-style=solid]
    0 & A & 0 & {} \\[-1em]
    0 & \Ddots & \Ddots & \Ddots \\[-1em]
    0 & \Ddots & & {} \\[-1em]
    & \Ddots & & {} \\
  \end{pNiceMatrix}
  \in \I  
  \),
  and similarly for $B,C$.
\item A more general context: given two sequences of (separately) mutually orthogonal projections, $\set{P_n}_{n=1}^{\infty}, \set{Q_n}_{n=1}^{\infty}$, for $T \in \I$ what can be said about ideal membership for $\sum_{n=1}^{\infty} Q_n T P_n$?
  In \Cref{thm:sum-off-diagonal-corners-am-closure} we establish that $\sum_{n=1}^{\infty} Q_n T P_n$ always lies in the arithmetic mean closure $\amclosure{\I}$ defined in \cite{DFWW-2004-AM} (see herein \cpageref{def:am-closed}).
  This follows from a generalization (see \Cref{thm:fans-theorem-pinching}) of Fan's famous submajorization theorem \cite[Theorem~1]{Fan-1951-PNASUSA} concerning partial sums of diagonals of operators.
\end{enumerate}

Throughout the paper we will prefer bi-infinite sequences (i.e., indexed by $\ints$ instead of $\nats$) of projections, but this is only to make the descriptions simpler;
we will not, however, use the term \term{bi-infinite} unless necessary for context.
The projections are allowed to be zero, so this is no restriction.
We first establish some terminology.

\begin{definition}
  \label{def:block-decomposition}
  A sequence $\set{P_n}_{n \in \ints}$ of mutually orthogonal projections $P_n \in \B(\Hil)$ for which $\sum P_n = I$ is a \term{block decomposition} and for $T \in \B(\Hil)$, partitions it into a (bi-)infinite matrix of operators $T_{i,j} := P_i T P_j$.

  We say that an operator $T$ is a \term{block band operator relative to $\set{P_n}$} if there is some $M \ge 0$, called the \term{block bandwidth}, for which $T_{i,j} = 0$ whenever $\abs{i - j} > M$.
  If $M=0$ (resp. $M=1$), we say $T$ is \emph{block diagonal (resp. block tridiagonal) relative to $\set{P_n}$}.
  
  Finally, in all the above definitions, if $\trace P_n \le 1$ for all $n \in \ints$, which, up to a choice of phase for each range vector, simply corresponds to a choice of orthonormal basis, then we omit the word ``block.''
  In this case, the operators $T_{i,j}$ are scalars and $\matrep{T_{i,j}}$ is the matrix representation (again, up to a choice of phase for each vector) for $T$ relative to this basis.

  If $\set{Q_n}_{n \in \ints}$ is an (unrelated) block decomposition, the pair $\set{P_n}_{n \in \ints}, \set{Q_n}_{n \in \ints}$ still determines a (bi-)infinite matrix of operators $T_{i,j} = Q_i T P_j$, but this time there is an inherent asymmetry in that $(T^{*})_{i,j} \neq (T_{j,i})^{*}$.
  In this case, the terms defined just above may be modified with the adjective ``asymmetric.''
\end{definition}

\begin{definition}
  \label{def:shift-representation}
  Suppose that $\set{P_n}_{n \in \ints}$ is a block decomposition for an operator $T \in \B(\Hil)$.
  For each $k \in \ints$, we call 
  \begin{equation*}
    T_k := \sum_{n \in \ints} T_{n,n+k} = \sum_{n \in \ints} P_n T P_{n+k}
  \end{equation*}
  the \term{$k^{\mathrm{th}}$ block diagonal} of $T$, which converges in the strong operator topology.
  Visually, these operators may be described with the following diagram\footnotemark{}:
  \begin{center}
    \includegraphics{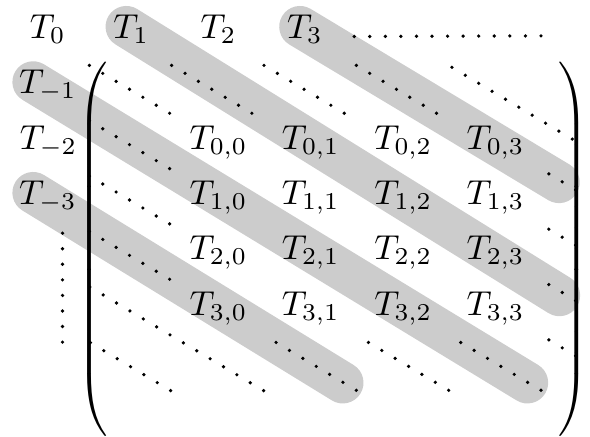}
  \end{center}
  \footnotetext{%
    For the case when the projections $P_n = 0$ for $n \in \ints \setminus \nats$, the matrix below is uni-infinite.
    This recovers uni-infinite matrix results from the bi-infinite approach we described in the paragraph preceding \Cref{def:block-decomposition}.%
  }%
  
  We call the collection $\set{T_k}_{k \in \ints}$ the \term{shift decomposition} of $T$ (relative to the block decomposition $\set{P_n}_{n \in \ints}$).
  The \term{asymmetric shift decomposition} $\set{T_k}_{k \in \ints}$ relative to \emph{different} block decompositions $\set{P_n}_{n \in \ints}, \set{Q_n}_{n \in \ints}$ is given by
  \begin{equation*}
    T_k := \sum_{n \in \ints} Q_n T P_{n+k}.
  \end{equation*}
  We note for future reference that sums of the above form don't require the sequences of projections to sum to the identity in order to converge in the strong operator topology, only that each sequence consists of mutually orthogonal projections.
  Moreover, it is elementary to show that when $T$ is compact, so is $T_k$ for all $k \in \ints$.
\end{definition}

\begin{remark}
  \label{rem:shift-decomposition-explanation}
  Although one has the formal equality $T = \sum_{k \in \ints} T_k$ in the sense that $T$ is uniquely determined by $\set{T_k}_{k \in \ints}$, this sum doesn't necessarily converge even in the weak operator topology \cite{Mer-1985-PAMS}, hence it doesn't converge in any of the usual operator topologies.
  If $\rank P_n = 1$ (and $Q_n = P_n$) for all $n \in \ints$ then $\sum_{k \in \ints} T_k$ does converge to $T$ in the \term{Bures topology}\footnotemark{} \cite{Bur-1971,Mer-1985-PAMS}.
  On the other hand, if $T$ is a block band operator relative to this block decomposition, then convergence is irrelevant: $T = \sum_{k=-M}^M T_k$.

  \footnotetext{%
    The Bures topology on $B(\Hil)$ is a locally convex topology constructed from the (rank-1) projections $P_n$ as follows.
    Let $\mathcal{D} = \bigoplus_{n \in \ints} P_n B(\Hil) P_n$ be the algebra of diagonal matrices and $E : B(\Hil) \to \mathcal{D}$ the conditional expectation given by $T \mapsto T_0 := \sum_{n \in \ints} P_n T P_n$.
    Then to each $\omega \in \ell_1 \cong \mathcal{D}_{*}$, associate the seminorm $T \mapsto \trace(\diag(\omega) E(T^{*}T)^{\frac{1}{2}})$, where $\diag : \ell_{\infty} \to \mathcal{D}$ is the natural *-isomorphism.
    These seminorms generate the Bures topology.
  }

  The reason for our ``shift'' terminology in \Cref{def:shift-representation} is that if the block decomposition $\set{P_n}_{n \in \ints}$ consists of rank-1 projections, then the operators $T_k$ have the form $T_k = U^k D_k$ where $D_k$ are diagonal operators and $U$ is the bilateral shift relative to any orthonormal basis corresponding to $\set{P_n}_{n \in \ints}$.
\end{remark}

\begin{remark}
  \label{rem:tridiagonalizability}
  All compact selfadjoint operators are diagonalizable via the spectral theorem.
  However, this is certainly not the case for arbitrary selfadjoint operators, the selfadjoint approximation theorem of Weyl--von Neumann notwithstanding.
  Nevertheless, every selfadjoint operator with a cyclic vector is \emph{tri}diagonalizable;
  for $T = T^{*}$ with cyclic vector $v$, apply Gram--Schmidt to the linearly independent spanning collection $\set{T^n v}_{n=0}^{\infty}$ and then $T$ is tridiagonal in the resulting orthonormal basis.
  Consequently, every selfadjoint operator is block diagonal with each nonzero block in the direct sum itself tridiagonal.

  The second author, along with Patnaik and Petrovic \cite{PPW-2020-TmloVLOt,PPW-2021}, recently established that every bounded operator is \emph{block} tridiagonalizable, meaning $T = T_{-1} + T_0 + T_1$, hence block banded (with block bandwidth $1$) and with finite block sizes growing no faster than exponential.
\end{remark}

Our first main theorem is an algebraic equality of ideals for block band operators relative to some block decomposition.

\begin{theorem}
  \label{thm:bandable}
  Let $T \in \B(\Hil)$ be an asymmetric block band operator of bandwidth $M$ relative to the block decompositions $\set{P_n}_{n \in \ints}, \set{Q_n}_{n \in \ints}$, and let $\set{T_k}_{k=-M}^M$ be the asymmetric shift decomposition of $T$.
  Then the following ideal equality holds:
  \begin{equation*}
    \ideal{T} = \sum_{k=-M}^M \ideal{T_k}.
  \end{equation*}
\end{theorem}

\begin{proof}
  The proof is essentially due to the following observation:
  if you zoom out and squint, then a band matrix looks diagonal.
  That is, we exploit the relative thinness of the diagonal strip of support entries.

  The ideal inclusion $\ideal{T} \subseteq \sum_{k=-M}^M \ideal{T_k}$ is obvious since $T = \sum_{k=-M}^M T_k$.
  Therefore it suffices to prove $T_k \in \ideal{T}$ for each $-M \le k \le M$.

  Indeed, for $-M \le j,k \le M$ define projections $R_{k,j} := \sum_{n \in \ints} P_{n(2M+1) + j + k}$ and $S_j := \sum_{n \in \ints} Q_{n(2M+1) + j}$
  Then whenever $n \neq m$, $Q_{n(2M+1) + j} T P_{m(2M+1) + j + k} = 0$ since the bandwidth of $T$ is $M$ and
  \begin{align*}
    \abs{\paren!{n(2M+1) + j} - \paren!{m(2M+1) + j + k} }
    &\ge \abs{n-m}(2M+1) - k \\
    &\ge (2M+1) - M > M.
  \end{align*}
  Therefore, for each $k,j$,
  \begin{equation*}
    S_j T R_{k,j} = \sum_{n \in \ints} Q_{n(2M+1) + j} T P_{n(2M+1) + j + k}
  \end{equation*}
  converges in the strong operator topology, and summing over $j$ yields
  \begin{equation*}
    \sum_{j=-M}^M S_j T R_{k,j} = \sum_{j=-M}^{M} \sum_{n \in \ints} Q_{n(2M+1) + j} T P_{n(2M+1) + j + k} = \sum_{n \in \ints} Q_n T P_{n+k} = T_k.
  \end{equation*}
  As a finite sum, the left-hand side is trivially in $\ideal{T}$ and therefore so is each $k^{\mathrm{th}}$ generalized block diagonal $T_k$.
\end{proof}

Before establishing our second main theorem (\Cref{thm:sum-off-diagonal-corners-am-closure}), we acquaint the reader with the prerequisite ideas concerning Fan's theorem \cite[Theorem~1]{Fan-1951-PNASUSA}, Hardy--Littlewood submajorization, fundamentals of the theory of operator ideals and arithmetic mean closed ideals, all of which are intimately related.

For a single operator, Fan's submajorization theorem \cite[Theorem~1]{Fan-1951-PNASUSA} states that if the matrix representation for a compact operator $T \in \K(\Hil)$ has diagonal sequence $\seq{d_j}_{j \in J}$ (with any index set $J$), then 
\begin{equation}
  \label{eq:submajorization}
  \sum_{n=1}^m \abs{d_n}^{*} \le \sum_{n=1}^m s_n(T) \quad\text{for all } m \in \nats,
\end{equation}
where $s(T) := \seq{s_n(T)}_{n \in \nats}$ denotes the (monotone) \term{singular value sequence} of $T$, and where $\seq{\abs{d_n}^{*}}_{n \in \nats}$ denotes the \term{monotonization}\footnotemark{} of the (possibly unordered) sequence $\seq{\abs{d_j}}_{j \in J}$;
the monotonization is always an element of the convex cone $\czstar$ of nonnegative nonincreasing sequences (indexed by $\nats$) converging to zero, even when $\seq{\abs{d_j}}_{j \in J}$ is indexed by another set $J$ different from $\nats$.
The set of inequalities \eqref{eq:submajorization} may be encapsulated, for pairs of sequences in $\czstar$, by saying that $\seq{\abs{d_n}^{*}}$ is \term{submajorized} by $s(T)$, which is often denoted $\seq{\abs{d_n}^{*}} \pprec s(T)$, although the precise notation for submajorization varies throughout the literature.
We remark the trivial fact that the submajorization order is finer than the usual pointwise order on $\czstar$;
that is, $\seq{a_n} \le \seq{b_n}$ implies $\seq{a_n} \pprec \seq{b_n}$ for any $\seq{a_n}, \seq{b_n} \in \czstar$.

\footnotetext{%
  This is the measure-theoretic \term{nonincreasing rearrangement} relative to the counting measure on the index set, say $J$, of $\seq{\abs{d_n}}$.
  Associated to this, there is a injection (not necessarily a bijection) $\pi : \nats \to J$ with $d^{-1}(\complex\setminus\set{0}) \subseteq \pi(\nats)$ such that $\abs{d_n}^{*} = \abs{d_{\pi(n)}}$.
  This of course requires $0 \notin (d \circ \pi)(\nats)$ when $d^{-1}(\complex\setminus\set{0})$ is infinite since $\seq{\abs{d_n}^{*}}$ is nonincreasing.
}

However, we view Fan's theorem in a slightly different way which is more amenable to our purposes.
In particular, consider the canonical trace-preserving conditional expectation\footnotemark{} $E : \B(\Hil) \to \mathcal{D}$ onto the masa (maximal abelian selfadjoint algebra) of diagonal operators relative to a fixed, but arbitrary, orthonormal basis.
Then the sequence $\seq{\abs{d_n}^{*}}$ is simply $s(E(T))$, and in this language:

\footnotetext{%
  For an inclusion of unital C*-algebras $\mathcal{B} \subseteq \mathcal{A}$ (with $1_{\mathcal{B}} = 1_{\mathcal{A}}$), a \term{conditional expectation of $\mathcal{A}$ onto $\mathcal{B}$} is a unital positive linear map $E : \mathcal{A} \to \mathcal{B}$ such that $E(bab') = bE(a)b'$ for all $a \in \mathcal{A}$ and $b,b' \in \mathcal{B}$.
  A conditional expectation is called \term{faithful} if $a \ge 0$ and $E(a) = 0$ imply $a = 0$.
  If $\mathcal{A}$ is a semifinite von Neumann algebra with a faithful normal semifinite trace $\tau$, then the expectation is said to be \term{trace-preserving} if $\tau(a) = \tau(E(a))$ for all $a \in \mathcal{A}_+$.%
}

\begin{theorem}[\protect{\cite[Theorem~1]{Fan-1951-PNASUSA}}]
  \label{thm:fans-theorem}
  If $T \in \K(\Hil)$ and $E : \B(\Hil) \to \mathcal{D}$ is the canonical conditional expectation onto a masa of diagonal operators, then
  \begin{equation*}
    s(E(T)) \pprec s(T),
  \end{equation*}
  that is, $s(E(T))$ is submajorized by $s(T)$.
\end{theorem}

The submajorization order features prominently in operator theory, but especially in the theory of diagonals of operators and in the related theory of operator ideals in $\B(\Hil)$.

For the reader's convenience we briefly review the basics of ideal theory.
Let $\czstar$ denote the convex cone of nonnegative nonincreasing sequences converging to zero.
To an ideal $\I$, Schatten \cite{Sch-1970}, in a manner quite similar to Calkin \cite{Cal-1941-AoM2}, associated the convex subcone $\charset(\I) := \setb{ s(T) \in \czstar }{ T \in \I }$, called the \term{characteristic set} of $\I$, which satisfies the properties:
\begin{enumerate}
\item If $\seq{a_n} \le \seq{b_n}$ (pointwise) and $\seq{b_n} \in \charset(\I)$, then $\seq{a_n} \in \charset(\I)$;
  that is, $\charset(\I)$ is a \term{hereditary subcone} of $\czstar$ with respect to the usual pointwise ordering.
\item If $\seq{a_n} \in \charset(\I)$, then $\seq{a_{\ceil{\frac{n}{2}}}} \in \charset(\I)$;
  that is, $\charset(\I)$ is closed under \term{$2$-ampliations}.
\end{enumerate}
Likewise, if $S$ is a hereditary (with respect to the pointwise order) convex subcone of $\czstar$ which is closed under $2$-ampliations, then $\I_S := \setb{ T \in \K(\Hil) }{ s(T) \in S }$ is an ideal of $\B(\Hil)$.
Finally, the maps $S \mapsto \I_S$ and $\I \mapsto \charset(\I)$ are inclusion-preserving inverses between the classes of $\B(\Hil)$-ideals and characteristic subsets of $\czstar$.

Ideals whose characteristic sets are also hereditary subcones with respect to the submajorization order (i.e., $B \in \I$ and $s(A) \pprec s(B)$ implies $A \in \I$) were introduced by Dykema, Figiel, Weiss and Wodzicki\fnmark{dfww} in \cite{DFWW-2004-AM} and are said to be \label{def:am-closed}\term{arithmetic mean closed}\footnotemark{} (abbreviated as \term{am-closed}).
Given an ideal $\I$, the smallest am-closed ideal containing $\I$ is called the am-closure, denoted $\amclosure{\I}$, and its characteristic set consists simply of the hereditary closure (with respect to the submajorization order) of $\charset(\I)$.
That is,
\begin{equation*}
  \charset\paren1{\amclosure{\I}} = \setb1{ \seq{a_n} \in \czstar }{ \exists \seq{b_n} \in \charset(\I), \seq{a_n} \pprec \seq{b_n} }.
\end{equation*}

In general, ideals are not am-closed.
Indeed, the sequence $\seq{1,0,0,\dots}$ corresponding to a rank-1 projection $P$ submajorizes any (nonnegative) sequence $\seq{a_n}$ whose sum is at most $1$.
Consequently, if $T \in \traceclass$, the trace class, then $s(T) \pprec s(\trace(\abs{T})P)$.
Therefore, since any ideal $\I$ contains the finite rank operators, if it is am-closed it must also contain the trace class $\traceclass$.
Additionally, it is immediate that $\traceclass$ is am-closed, making it the minimum am-closed ideal.

\footnotetext[\arabic{dfww}]{%
  The description given \cite{DFWW-2004-AM} is not in terms of the submajorization order, but these two definitions are easily shown to be equivalent.
  Instead, for an ideal $\I$, \cite{DFWW-2004-AM} defines the \term{arithmetic mean ideal} $\I_a$ and \term{pre-arithmetic mean ideal} ${}_a\I$ whose characteristic sets are given by
  \begin{gather*}
    \charset(\I_a) := \setb*{ \seq{a_n} \in \czstar }{ \exists \seq{b_n} \in \charset(\I), a_n \le \frac{1}{n} \sum_{k=1}^n b_k } \\
    \charset({}_a\I) := \setb*{ \seq{a_n} \in \czstar }{ \exists \seq{b_n} \in \charset(\I), \frac{1}{n} \sum_{k=1}^n a_k \le b_n } 
  \end{gather*}
  Then the \term{arithmetic mean closure} of $\I$ is $\amclosure{\I} := {}_a(\I_a)$, and $\I$ is called \term{am-closed} if $\I = \amclosure{\I}$.
  This viewpoint also allows one to define the \term{arithmetic mean interior} $({}_a\I)_a$, and one always has the inclusions ${}_a\I \subseteq ({}_a\I)_a \subseteq \I \subseteq {}_a(\I_a) \subseteq \I_a$.
}

\footnotetext{%
  Although am-closed ideals were introduced in this generality by \cite{DFWW-2004-AM}, they had been studied at least as early as \cite{GK-1969-ITTTOLNO,Rus-1969-FAA}, but only in the context of \term{symmetrically normed ideals}.
  In the study of symmetrically normed ideals by Gohberg and Krein \cite{GK-1969-ITTTOLNO}, they only considered those which were already am-closed, but they did not have any terminology associated to this concept.
  
  Around the same time, both Mityagin \cite{Mit-1964-IANSSM} and Russu \cite{Rus-1969-FAA} concerned themselves with the existence of so-called \term{intermediate} symmetrically normed ideals, which are necessarily not am-closed, or in the language of Russu, do not possess the \term{majorant property}.
  In \cite{Rus-1969-FAA}, Russu also established that the majorant property is equivalent to the \term{interpolation property} studied by Mityagin \cite{Mit-1965-MSNS} and Calder\'on \cite{Cal-1966-SM}.
  
  In the modern theory of symmetrically normed ideals, those which are am-closed (equivalently, have the majorant or interpolation properties), are said to be \term{fully symmetric}, but this term also implies the norm preserves the submajorization order.
  For more information on fully symmetrically normed ideals and related topics, we refer the reader to \cite{LSZ-2013-STTAA}.%
}

Arithmetic mean closed ideals are important within the lattice of operator ideals not least for their connection to Fan's theorem, but also because of the following sort of converse due to the second author with Kaftal.

\begin{theorem}[\protect{\cite[Corollaries~4.4,~4.5]{KW-2011-IUMJ}}]
  \label{thm:diagonal-invariance}
  For an operator ideal $\I$, and the canonical conditional expectation $E : \B(\Hil) \to \mathcal{D}$ onto a masa of diagonal operators,
  \begin{equation*}
    E(\I) = \amclosure{\I} \cap \mathcal{D}.
  \end{equation*}
  Consequently, $\I$ is am-closed if and only if $E(\I) \subseteq \I$.
\end{theorem}

They used the term \term{diagonal invariance} to refer to $E(\I) \subseteq \I$, and so $\I$ is am-closed if and only if it is diagonally invariant.
The reader should note that the inclusion $E(\I) \subseteq \amclosure{\I} \cap \mathcal{D}$ is a direct consequence of Fan's theorem, when viewed through the lens of \Cref{thm:fans-theorem}, so the new content of \Cref{thm:diagonal-invariance} lies primarily in the reverse inclusion.

At this point, we note an important contrapositive consequence of \Cref{thm:bandable} and \Cref{thm:diagonal-invariance}.
Suppose $T$ is positive and $\ideal{T}$ is not am-closed, then by \Cref{thm:diagonal-invariance} there is some basis in which the main diagonal of $T$ does not lie in $\ideal{T}$, and therefore by \Cref{thm:bandable}, $T$ is not a band operator in this basis.

The next theorem, due originally to Gohberg--Krein \cite[Theorems~II.5.1 and III.4.2]{GK-1969-ITTTOLNO}, bootstraps \Cref{thm:fans-theorem} to apply to conditional expectations onto block diagonal algebras instead of simply diagonal masas.
We include this more modern proof both for completeness and to make the statement accord with that of \Cref{thm:fans-theorem}.

\begin{theorem}
  \label{thm:fans-theorem-pinching}
  Let $\mathcal{P} = \set{P_n}_{n \in \ints}$ be a block decomposition and consider the associated conditional expectation $\ep : \B(\Hil) \to \bigoplus_{n \in \ints} P_n \B(\Hil) P_n$ defined by $\ep(T) := T_0 = \sum_{n \in \ints} P_n T P_n$.
  If $T \in \K(\Hil)$, then $s(\ep(T))$ is submajorized by $s(T)$, i.e.,
  \begin{equation*}
    s(\ep(T)) \pprec s(T).
  \end{equation*}
  Moreover, if $T \in \I$, then $\ep(T) \in \amclosure{\I}$.
  In addition, if $s(\ep(T)) = s(T)$, then $\ep(T) = T$.
\end{theorem}

\begin{proof}
  Suppose that $\mathcal{D}$ is a diagonal masa contained in the algebra $\bigoplus_{n \in \ints} P_n \B(\Hil) P_n$, and let $E : \B(\Hil) \to \mathcal{D}$ be the associated canonical trace-preserving conditional expectation.
  Because of the algebra inclusions, we see that $E \circ \ep = E$.

  Let $T \in \K(\Hil)$ and consider $\ep(T)$.
  By applying the Schmidt decomposition to each $P_n T P_n$ one obtains partial isometries $U_n, V_n$ (the latter may even be chosen unitary) in $P_n \B(\Hil) P_n$ so that $U_n P_n T P_n V_n$ is a positive operator in $\mathcal{D}$.
  Then $U := \bigoplus_{n \in \ints} U_n, V := \bigoplus_{n \in \ints} V_n$ are partial isometries for which $s((E(U \ep(T) V)) = s(\ep(T))$.
  Then since $U,V \in \bigoplus_{n \in \ints} P_n \B(\Hil) P_n$ they commute with the conditional expectation $\ep$ and hence
  \begin{equation*}
    s(\ep(T)) = s((E(U \ep(T) V)) = s(E(\ep(UTV))) = s(E(UTV)).
  \end{equation*}
  By Fan's theorem (\Cref{thm:fans-theorem}), $s(E(UTV)) \pprec s(UTV) \le \norm{U} s(T) \norm{V} = s(T)$, and therefore $s(\ep(T)) \pprec s(T)$.
  Finally, this fact along with the definition of the arithmetic mean closure guarantees $T \in \I$ implies $\ep(T) \in \amclosure{\I}$.

  For the case of equality, now suppose that $s(\ep(T)) = s(T)$.
  Let $\set{e_n}_{n \in \nats}$ be an orthonormal sequence of eigenvectors of $\ep(T)^{*} \ep(T)$, each of which is inside one of the subspaces $P_j \Hil$,  satisfying $\ep(T)^{*} \ep(T) e_n = s_n(T)^2 e_n$.
  Then the projections $Q_n$ onto $\spans\set{e_1,\ldots,e_n}$ commute with each $P_j$, and hence also with the expectation $\ep$.
  We note for later reference that
  \begin{equation}
    \label{eq:epTQnperp}
    \norm{\ep(T)Q_n^{\perp}}^2 = \norm{Q_n^{\perp}\ep(T)^{*}\ep(T) Q_n^{\perp}} \le s_{n+1}(\ep(T))^2.
  \end{equation}

  Observe that for any operator $X$, because $P_j X^{*} P_j X P_j \le P_j X^{*} X P_j$
  \begin{equation}
    \label{eq:epX}
    \ep(X)^{*}\ep(X) = \sum_{j \in \ints} P_j X^{*} P_j X P_j \le \sum_{j \in \ints} P_j X^{*} X P_j = \ep(X^{*}X),
  \end{equation}
  with equality if and only if $P_j X^{*} P_j^{\perp} X P_j = 0$ for all $j \in \ints$ if and only if $P_j^{\perp} X P_j = 0$ for all $j \in \ints$ if and only if $X = \ep(X)$.

  Applying \eqref{eq:epX} to $X = TQ_n$,
  \begin{align*}
    \sum_{j=1}^n s_j(\ep(T))^2 &= \trace(Q_n \ep(T)^{*}\ep(T) Q_n) \\
                               &= \trace(\ep(TQ_n)^{*} \ep(TQ_n)) \\
                               &\le \trace(\ep(Q_n T^{*} T Q_n)) \\
                               &= \trace(Q_n T^{*} T Q_n) \le \sum_{j=1}^n s_j(T)^2, 
  \end{align*}
  where the last inequality follows from \Cref{thm:fans-theorem}.
  We must have equality throughout since $s(\ep(T)) = s(T)$.
  Consequently, $TQ_n = \ep(TQ_n) = \ep(T)Q_n$ for all $n \in \nats$ by the equality case of \eqref{eq:epX}.

  By construction, $\norm{\ep(T)Q_n^{\perp}} \to 0$ as $n \to \infty$, but we also claim
  \begin{equation}
    \label{eq:TQnperp}
    \norm{TQ_n^{\perp}} \le s_{n+1}(T).
  \end{equation}

  Suppose not.
  Then we could find some unit vector $x \in Q_n^{\perp} \Hil$ with $\innerprod{T^{*}T x}{x} = \norm{T x}^2 > s_{n+1}(T)^2$, and therefore, for the projection $R = Q_n + (x \otimes x)$,
  \begin{equation*}
    \trace(RT^{*}TR) = \trace(Q_n T^{*}T Q_n) + \innerprod{T^{*}T x}{x} > \sum_{j=1}^{n+1} s_j(T)^2,
  \end{equation*}
  contradicting the fact that, because $R$ is a projection of rank $n+1$, by \Cref{thm:fans-theorem}
  \begin{equation*}
    \trace(RT^{*}TR) \le \sum_{j=1}^{n+1} s_j(RT^{*}TR) \le \sum_{j=1}^{n+1} s_j(T)^2.
  \end{equation*}

  Finally, again noting that $TQ_n = \ep(T)Q_n$,
  \begin{equation*}
    0 \le \norm{T - \ep(T)} \le \norm{T - TQ_n} + \norm{\ep(T)Q_n - \ep(T)} = \norm{TQ_n^{\perp}} + \norm{\ep(T)Q_n^{\perp}}.
  \end{equation*}
  Since $\norm{TQ_n^{\perp}} \le s_{n+1}(T)$ by \eqref{eq:TQnperp} and $\norm{\ep(T)Q_n^{\perp}} \le s_{n+1}(\ep(T))$ by \eqref{eq:epTQnperp}, the right-hand side converges to zero as $n \to \infty$.
  Therefore, $\norm{T - \ep(T)} = 0$ and hence $T = \ep(T)$.
\end{proof}.

\begin{remark}
  \label{rem:T=ep(T)-hilbert-schmidt}
  When $T$ is Hilbert--Schmidt, the proof that $s(\ep(T)) = s(T)$ implies $\ep(T) = T$ may be shortened considerably.
  In particular, $T^{*}T, \ep(T)^{*}\ep(T)$ are trace-class with $s(T^{*}T) = s(\ep(T)^{*}\ep(T))$ and so $\trace(T^{*}T) = \trace(\ep(T)^{*}\ep(T))$.
  Since the expectation $\ep(T)$ is trace-preserving,
  \begin{align*}
    \trace(\ep(T^{*}T) - \ep(T)^{*}\ep(T)) &= \trace(\ep(T^{*}T - \ep(T)^{*}\ep(T))) \\
    &= \trace(T^{*}T - \ep(T)^{*}\ep(T)) = 0.
  \end{align*}
  Since $\ep(T^{*}T) - \ep(T)^{*}\ep(T)$ is a positive operator by \eqref{eq:epX} and the trace is faithful, we must have $\ep(T^{*}T) = \ep(T)^{*}\ep(T)$, and hence $T = \ep(T)$ by the equality case of \eqref{eq:epX}.
\end{remark}

\begin{remark}
  Fan's theorem (\Cref{thm:fans-theorem}) is a special case of \Cref{thm:fans-theorem-pinching} by selecting the projections $P_n$ to have rank one, and therefore $E = \ep$.
  As we need \Cref{thm:fans-theorem} to prove \Cref{thm:fans-theorem-pinching}, this doesn't provide an independent proof of Fan's theorem.
\end{remark}

Our second main theorem says that there is nothing special about the main diagonal $T_0$: for all $k \in \ints$, $s(T_k) \pprec s(T)$;
Moreover, this holds even for \emph{asymmetric} shift decompositions.

\begin{theorem}
  \label{thm:sum-off-diagonal-corners-am-closure}
  Suppose that $\set{P_n}_{n \in \ints}, \set{Q_n}_{n \in \ints}$ are block decompositions and let $T \in \K(\Hil)$ with asymmetric shift decomposition $\set{T_k}_{k \in \ints}$.
  Then $s(T_k) \pprec s(T)$.
  Consequently, if $T$ lies in some ideal $\I$, then $T_k \in \amclosure{\I}$.
\end{theorem}

\begin{proof}
  It suffices to prove the theorem for $T_0$ since $T_k$ is simply $T_0$ relative to the translated block decomposition pair $\set{P_{n+k}}_{n \in \ints}, \set{Q_n}_{n \in \ints}$.

  Each $Q_n T P_n$ has the polar decomposition $Q_n T P_n = U_n \abs{Q_n T P_n}$  where $U_n$ is a partial isometry\footnotemark{} with $Q_n U_n = U_n = U_n P_n$.
  Then $U := \sum_{n \in \ints} U_n$ converges in the strong operator topology since the collections $\set{P_n}_{n \in \ints}, \set{Q_n}_{n \in \ints}$ are each mutually orthogonal and hence also $U$ is a partial isometry.
  Moreover,
  \begin{equation*}
    T_0^{*} T_0 = \paren*{ \sum_{n \in \ints} P_n T^{*} Q_n } \paren*{ \sum_{m \in \ints} Q_m T P_m } = \sum_{n \in \ints} \abs{Q_n T P_n}^2.
  \end{equation*}
  Since the operators $\abs{Q_n T P_n}^2$ are orthogonal (i.e., their products are zero), $\abs{T_0} = (T_0^{*}T_0)^{\frac{1}{2}} = \sum_{n \in \ints} \abs{Q_n T P_n}$.
  Thus,
  \begin{align*}
    \ep(U^{*}T) &= \sum_{n \in \ints} P_n U^{*} T P_n = \sum_{n \in \ints} \paren*{ \sum_{m \in \ints} P_n U^{*}_m T P_n }  \\
             &= \sum_{n \in \ints} \paren*{ \sum_{m \in \ints} P_n P_m U^{*}_m Q_m T P_n } = \sum_{n \in \ints} U^{*}_n Q_n T P_n  \\
             &= \sum_{n \in \ints} \abs{Q_n T P_n} = \abs{T_0}.
  \end{align*}
  Finally, by \Cref{thm:fans-theorem-pinching} and since $U^{*}$ is a contraction,
  \begin{equation*}
    s(T_0) = s(\abs{T_0}) = s(\ep(U^{*}T)) \pprec s(U^{*}T) \le s(T).
  \end{equation*}
  Therefore, if $T \in \I$, then $T_0 \in \amclosure{\I}$ by definition.
\end{proof}

\footnotetext{That $Q_n U_n = U_n = U_n P_n$ follows from well-known facts (e.g., see \cite[Theorem~I.8.1]{Dav-1996}) when $U_n$ is taken to be the canonical unique partial isometry on $\Hil$ mapping $\closure{\range(\abs{Q_n T P_n})} \to \closure{\range(Q_n TP_n)}$ and noting also the range projection of $Q_n T P_n$ is dominated by $Q_n$ and the projection onto $\closure{\range(\abs{Q_n T P_n})} = \ker^{\perp}(Q_n T P_n)$ is dominated by $P_n$.}

\begin{remark}
  In the previous theorem we assumed that $\set{P_n}_{n \in \ints}, \set{Q_n}_{n \in \ints}$ were block decompositions, but the condition that they sum to the identity is not actually necessary (the same proof given above still works).
  Therefore, if $\set{P_n}_{n \in \ints}, \set{Q_n}_{n \in \ints}$ are sequences of mutually orthogonal projections then $s(\sum_{n \in \ints} Q_n T P_n) \pprec s(T)$;
  consequently, if $T \in \I$ then we still have $\sum_{n \in \ints} Q_n T P_n \in \closure[\mathrlap{am}]{\I}$.
\end{remark}

\section*{Acknowledgments}

The authors would like to thank Fedor Sukochev for providing insight into the history of fully symmetric ideals.

\printbibliography

\end{document}